\documentclass[a4paper,10pt]{article}
\usepackage[utf8]{inputenc}
\usepackage[T1]{fontenc}
\usepackage[english]{babel}
\usepackage{amsmath,amsthm,amssymb}
\usepackage{url}
\usepackage{eucal}
\usepackage{esint}
\usepackage{bm}
\usepackage{fancyhdr}
\usepackage{color}
\usepackage{tikz}
\usepackage{graphicx}
\usepackage[font=small]{caption}

\newtheorem{theorem}{Theorem}
\newtheorem{proposition}[theorem]{Proposition}
\newtheorem{lemma}[theorem]{Lemma}

\theoremstyle{definition}
\newtheorem{remark}[theorem]{Remark}
\newtheorem{example}[theorem]{Example}
\numberwithin{equation}{section}

\newcommand{\ddiv}{\operatorname{div}}
\newcommand{\uHt}{\tilde{u}_H}
\newcommand{\smooth}{^{\mathit{reg}}}
\newcommand{\loc}{^{(\ell)}}
\newcommand{\sym}{\operatorname{sym}}
\newcommand{\card}{\operatorname{card}}

\newcommand{\T}{\mathcal{T}}
\newcommand{\Cor}{\mathcal{C}}
\newcommand{\nei}{\mathsf{N}}
\newcommand{\wcba}{\bm{\mathrm{wcba}}}

\begin{document}
\pagestyle{fancy}
\fancyhead{}
\setlength{\headheight}{14pt}
\renewcommand{\headrulewidth}{0pt}

\fancyhead[c]{\small \it Numerical homogenization}
\title{Computation of quasilocal effective
       diffusion tensors and connections to the mathematical theory of homogenization}

\author{%
       D.~Gallistl\thanks{Institut f\"ur Angewandte und Numerische Mathematik,
         Karlsruher Institut f\"ur Technologie,
         76049 Karlsruhe, Germany
         }\;\;\footnotemark[3]
        \and 
       D.~Peterseim\thanks{Institut f\"ur Mathematik,
         Universit\"at Augsburg,        
         86135 Augsburg, Germany
         }\;\;\thanks{D.~Gallistl is supported by Deutsche Forschungsgemeinschaft (DFG) through CRC 1173.
                       D.~Peterseim acknowledges support by DFG in the Priority Program 1748 ``Reliable simulation techniques in solid mechanics'' (PE2143/2-1).
                       Main parts of this paper were written at the Institut f\"ur Numerische
                       Simulation (Bonn). The authors also thank
                       the Hausdorff Institute for Mathematics in Bonn
                       for the kind hospitality of during the trimester program on multiscale problems. }
       }
\date{}

\maketitle

\begin{abstract}
This paper aims at bridging existing theories in
numerical and analytical homogenization.
For this purpose
the multiscale method
of M{\aa}lqvist and Peterseim [Math.\ Comp.\ 2014],
which is based on orthogonal subspace decomposition,
is reinterpreted by means
of a discrete integral operator acting on standard finite element
spaces.
The exponential decay of the involved integral kernel
motivates the use of a diagonal approximation and, hence, 
a localized piecewise constant coefficient.
In a periodic setting, the computed localized coefficient is proved to
coincide with the classical homogenization limit.
An a~priori error analysis shows that the local numerical model is
appropriate beyond the periodic setting when the localized coefficient 
satisfies a certain homogenization criterion, which can be verified 
a~posteriori. The results are illustrated in numerical experiments.
\end{abstract}

{\small
\noindent
\textbf{Keywords}
numerical homogenization, multiscale method, upscaling,
a~priori error estimates, a~posteriori error estimates

\noindent
\textbf{AMS subject classification}
65N12,  
65N15,  
65N30,  
73B27,  
74Q05  
}

\section{Introduction}

Consider the prototypical elliptic model problem
\begin{equation*}
-\ddiv A_\varepsilon \nabla u = f
\end{equation*}
where the diffusion coefficient $A_\varepsilon$ encodes
microscopic features on some characteristic length scale $\varepsilon$.
Homogenization is a tool of mathematical modeling to identify 
reduced descriptions of the macroscopic response of such
multiscale models in the limit as $\varepsilon$ tends to zero. 
It turns out that suitable limits represented by the so-called 
effective or homogenized coefficient exist in fairly general 
settings in the framework of $G$-, $H$-, or two-scale convergence
\cite{Spagnolo:1968,Gio75,Mur78,MR990867,MR1185639}.
In general, the effective coefficient is not explicitly given 
but is rather the result of an implicit representation based on 
cell problems.
This representation
usually requires structural assumptions on the 
sequence of coefficients $A_\varepsilon$ such as local periodicity 
and scale separation \cite{BLP78}. Under such assumptions, 
efficient numerical methods for the approximate evaluation of 
the homogenized model are available, e.g., the Heterogeneous Multiscale Method (HMM) 
\cite{EE03,AbdulleEEngquistVandenEijnden2012} or the two-scale finite element method \cite{MaSch02}.

In contrast to this idealized setting of analytical homogenization,
in practice one
is often concerned with one coefficient $A$ with heterogeneities 
on multiple nonseperable scales and a corresponding sequence of 
scalable models can hardly be identified or may not be available at all.
That is why we are interested in the computation of effective 
representations of very rough coefficients beyond structural 
assumptions such as scale separation and local periodicity. 
In recent years, many  numerical attempts have been developed 
that conceptually do not rely on analytical homogenization results
for rough cases. Amongst them are the multiscale finite element 
method \cite{MR1455261,EfendievHou2009}, metric-based upscaling
\cite{MR2292954}, hierarchical matrix compression 
\cite{GreffHackbusch2008,Hackbusch2015}, 
the flux-norm approach \cite{BO10}, 
generalized finite elements based on spectral cell problems
\cite{BabuskaLipton2011,Efendiev2013116}, 
the AL basis \cite{GrasedyckGreffSauter2012,WeymuthSauter2015},
rough polyharmonic splines \cite{OwhadiZhangBerlyand2013},
iterative numerical homogenization \cite{KornhuberYserentant2015}, and 
gamblets \cite{2015arXiv150303467O}. 
 
Another construction based on concepts of orthogonal subspace 
decomposition and the 
solution of localized microscopic cell problems was given
in \cite{MalqvistPeterseim2014} 
and later optimized in
\cite{HenningPeterseim2013,HenningMorgensternPeterseim2015,GallistlPeterseim2015,Peterseim:2015}.
The method is referred to as the Localized Orthogonal Decomposition
(LOD) method.
The approach is inspired by ideas of the 
variational multiscale method
\cite{MR1660141,HughesSangalli2007,Malqvist:2011}.
As most of the methods above, the LOD
constructs a basis representation of some 
finite-dimensional operator-dependent subspace with superior 
approximation properties
rather than computing an upscaled coefficient. The effective 
model is then a discrete one represented by the corresponding 
stiffness matrix and possibly right-hand side. 
The computation of an effective coefficient is, however,
often favorable and this paper re-interprets and modifies the
LOD method in this regard. 

To this end, we revisit the LOD method of
\cite{MalqvistPeterseim2014}.
The original method employs finite element basis functions that
are modified by a fine-scale correction with a slightly larger
support.
We show that it is possible to rewrite the method 
by means of a discrete integral operator  acting on standard
finite element spaces.
The discrete operator is of non-local nature
and is not necessarily associated with a
differential operator on the energy space
$H^1_0(\Omega)$ (for the physical domain $\Omega\subseteq\mathbb R^d$).
The observation scale $H$ is associated with some quasi-uniform
mesh $\T_H$ of width $H$.
We are able to show that there is a discrete effective
non-local model represented by an integral kernel
$\mathcal A_H\in L^\infty(\Omega\times\Omega;\mathbb R^{d\times d})$
such that the problem is well-posed on a finite-element space
$V_H$ with similar constants
and satisfies
\begin{equation*}
\sup_{f\in L^2(\Omega)\setminus\{0\}}
\frac{\|u(f)-u_H(f)\|_{L^2(\Omega)}}{\|f\|_{L^2(\Omega)}}
\lesssim
\sup_{f\in L^2(\Omega)\setminus\{0\}} \inf_{v_H\in V_H}
 \frac{\| u(f)-v_H\|_{L^2(\Omega)}}{\|f\|_{L^2(\Omega)}} + H^2.
\end{equation*}
Based on the exponential decay of that kernel
$\mathcal A_H$ away from the 
diagonal, we propose a quasi-local and sparse formulation 
as an approximation. The storage requirement for the quasi-local kernel 
is $\mathcal O(H^{-d} \lvert\log H\rvert)$.

For an even stronger compression to $\mathcal O(H^{-d})$ information, 
one can replace $\mathcal A_H$
by a local and piecewise constant tensor field $A_H$.
It turns out that this localized effective coefficient $A_H$ coincides with 
the homogenized coefficient of classical homogenization in the periodic case
provided that the structure of the coefficient is slightly
stronger than only periodic and that the mesh is suitably aligned with the periodicity pattern. 
In this sense, the results of this paper bridge the multiscale method 
of \cite{MalqvistPeterseim2014} with classical analytical techniques and numerical methods such as HMM.
With regard to the recent reinterpretation of the multiscale 
method in \cite{Kornhuber.Peterseim.Yserentant:2016}, the paper even 
connects all the way from analytic homogenization to the theory of 
iterative solvers.

This new representation of the multiscale method turns out to be
particularly attractive for computational stochastic homogenization
\cite{GallistlPeterseim2017}.
A further advantage of our numerical techniques when compared with classical analytical techniques is that they are still applicable in the general non-periodic case, where the local numerical model yields
reasonable results whenever a certain quantitative homogenization criterion
is satisfied, which can be checked a~posteriori through
a computable model error estimator. Almost optimal convergence rates
can be proved under reasonable assumptions on the data. 

\medskip
The structure of this article is as follows.
After the preliminaries on the model problem and notation 
from Section~\ref{s:notation},
we review the LOD method of \cite{MalqvistPeterseim2014}
and introduce the quasi-local effective discrete coefficients
in Section~\ref{s:nonloc}.
In Section~\ref{s:local}, we present the error analysis for
the localized effective coefficient.
Section~\ref{s:periodic} studies the particular case of a
periodic coefficient.
We present numerical results in Section~\ref{s:num}.

\bigskip
Standard notation on Lebesgue and Sobolev spaces
applies throughout this paper. 
The notation $a\lesssim b$ abbreviates $a\leq C b$ for some
constant $C$ that is independent of the mesh-size,
but may depend on the contrast of the coefficient $A$;
$a\approx b$ abbreviates $a\lesssim b\lesssim a$.
The symmetric part of a quadratic matrix $M$ is
denoted by $\sym(M)$.

\section{Model problem and notation}\label{s:notation}

This section describes the model problem and some notation
on finite element spaces.
\subsection{Model problem}
Let $\Omega\subseteq\mathbb R^d$ for $d\in\{1,2,3\}$ be
an open Lipschitz polytope.
We consider the prototypical model problem
\begin{equation}\label{e:modelstrong}
-\ddiv (A \nabla u ) = f
\quad\text{in }\Omega,
\qquad
u|_{\partial\Omega}=0.
\end{equation}
The coefficient $A\in L^\infty(\Omega;\mathbb R^{d\times d})$ 
is assumed to be symmetric
and to satisfy the following uniform
spectral bounds
\begin{equation}\label{e:Abounds}
0<\alpha
\leq \operatorname*{ess\,inf}_{x\in\Omega}
     \inf_{\xi\in\mathbb R^d\setminus\{0\}}
    \frac{\xi \cdot ( A(x)\xi)}{\xi\cdot\xi}
\leq
\operatorname*{ess\,sup}_{x\in\Omega}
\sup_{\xi\in\mathbb R^d\setminus\{0\}}
\frac{\xi\cdot ( A(x)\xi)}{\xi\cdot\xi}
\leq \beta.
\end{equation}
The symmetry of $A$ is not essential for our analysis
and is assumed for simpler notation.
The weak form employs the Sobolev space
$V:=H^1_0(\Omega)$
and the bilinear form $a$ defined,
for any $v,w\in V$, by
\begin{equation*}
a(v,w) := (A\nabla v,\nabla w)_{L^2(\Omega)}.
\end{equation*}
Given $f\in L^2(\Omega)$ and the linear functional
$$
 F:V\to\mathbb R,
 \quad\text{with }
 F(v):= \int_\Omega fv\,dx
 \quad\text{ for any } v\in V,
$$
the weak form seeks $u\in V$ such that
\begin{equation}\label{e:modelweak}
a(u,v) = F(v) \quad\text{for all } v\in V.
\end{equation}

\subsection{Finite element spaces}
Let $\T_H$ be a quasi-uniform regular triangulation of $\Omega$
and let $V_H$ denote the standard $P_1$ finite element space,
that is, the subspace of $V$ consisting of piecewise first-order
polynomials.

Given any subdomain $S\subseteq\overline\Omega$, define its neighbourhood
via
\begin{equation*}
\nei(S):=\operatorname{int}
          \Big(\cup\{T\in\T_H\,:\,T\cap\overline S\neq\emptyset  \}\Big).
\end{equation*}
Furthermore, we introduce for any $m\geq 2$ the patch extensions
\begin{equation*}
\nei^1(S):=\nei(S)
\qquad\text{and}\qquad
\nei^m(S):=\nei(\nei^{m-1}(S)) .
\end{equation*}
Throughout this paper, we assume that the coarse-scale mesh $\T_H$
belongs to a family of quasi-uniform triangulations.
The global mesh-size reads 
$H:=\max\{\operatorname{diam}(T):T\in\T_H\}$.
Note that the shape-regularity implies that there is a uniform bound 
$C(m)$
on the number of elements in the $m$th-order patch,
$
\card\{ K\in\T_H\,:\, K\subseteq \overline{\nei^m(T)}\}
\leq C(m)
$
for all ${T\in\T_H}$.
The constant $C(m)$ depends polynomially on $m$.
The set of interior $(d-1)$-dimensional hyper-faces 
of $\T_H$ is denoted by
$\mathcal F_H$. For a piecewise continuous function $\varphi$,
we denote the jump across an interior edge by $[\varphi]_F$,
where the index $F$ will be sometimes omitted for brevity.
The space of piecewise constant $d\times d$ matrix fields
is denoted by
$P_0(\T_H;\mathbb R^{d\times d})$.

Let $I_H:V\to V_H$ be a 
surjective
quasi-interpolation operator that
acts as a $H^1$-stable and $L^2$-stable
quasi-local projection in the sense that
$I_H\circ I_H = I_H$ and that
for any $T\in\T_H$ and all $v\in V$ there holds
\begin{align}
\label{e:IHapproxstab}
H^{-1}\|v-I_H v\|_{L^2(T)} + \|\nabla I_H v \|_{L^2(T)}
&
\leq C_{I_H} \|\nabla v\|_{L^2(\nei(T))} 
\\
\label{e:IHstabL2}
\|I_H v\|_{L^2(T)}
&
\leq C_{I_H} \|v\|_{L^2(\nei(T))} .
\end{align}
Since $I_H$ is a stable projection from $V$ to $V_H$,
any $v\in V$ is quasi-optimally approximated by $I_H v$
in the $L^2(\Omega)$ norm as well as in the $H^1(\Omega)$ norm.
One possible choice
is to define $I_H:=E_H\circ\Pi_H$, where
$\Pi_H$ is the $L^2$ projection onto 
the space $P_1(\T_H)$ of piecewise affine (possibly discontinuous)
functions
and $E_H$ is the averaging operator that maps $P_1(\T_H)$ to $V_H$ by
assigning to each free vertex the arithmetic mean of the corresponding
function values of the neighbouring cells, that is, for any $v\in P_1(\T_H)$
and any free vertex $z$ of $\T_H$,
\begin{equation}\label{e:EHdef}
(E_H(v))(z) =
           \sum_{\substack{T\in\T_H\\\text{with }z\in T}}v|_T (z) 
           \bigg/
           \card\{K\in\T_H\,:\,z\in K\}.
\end{equation}
This choice of $I_H$ is employed in our numerical experiments.

\section{Non-local effective coefficient}\label{s:nonloc}

We introduce a modified version of 
the LOD method of 
\cite{MalqvistPeterseim2014,HenningPeterseim2013}
and its localization.
We give a new interpretation by means of a non-local effective
coefficient and present an a~priori error estimate.

\subsection{A modified LOD method}

Let $W:=\operatorname{ker}I_H\subseteq V$ denote the kernel of $I_H$.
Given any $T\in\T_H$ and $j\in\{1,\dots,d\}$,
the element corrector $q_{T,j}\in W$ is the solution of
the variational problem
\begin{equation}\label{e:qTelementcorr}
 a( w, q_{T,j} )
 =  \int_T \nabla w\cdot( A e_j)\, dx 
 \quad \text{for all }w\in W.
\end{equation}
Here $e_j$ is the $j$-th standard Cartesian
unit vector in $\mathbb{R}^d$. 
The gradient of any $v_H\in V_H$ has the representation
\begin{equation*}
\nabla v_H = \sum_{T\in\T_H} \sum_{j=1}^d (\partial_j v_H|_T) e_j.
\end{equation*}
Given any $v_H\in V_H$, define the corrector $\Cor v_H$ by
\begin{equation}\label{e:CorvHexpansion}
\Cor v_H = \sum_{T\in\T_H} \sum_{j=1}^d (\partial_j v_H|_T) q_{T,j} .
\end{equation}
We remark that for any $v_H\in V_H$ the gradient
$\nabla v_H$ is piecewise constant and, thus,
$\Cor v_H$ is a finite linear combination of the 
element correctors $q_{T,j}$.
It is readily verified that, for any $v_H\in V_H$,
$\Cor v_H$ is the $a$-orthogonal projection on $W$, i.e.,
\begin{equation}\label{e:cororth}
a(w,v_H-\Cor v_H)=0\quad\text{for all }w\in W.
\end{equation}
Clearly, by \eqref{e:cororth}, the projection $\Cor v\in W$ is
well-defined for any $v\in V$. The representation
\eqref{e:CorvHexpansion} for discrete functions will, however,
be useful in this work.

The LOD method in its version from \cite{MalqvistPeterseim2014} seeks
$\bar u_H\in V_H$ such that
\begin{equation}\label{e:msprimal}
a((1-\Cor)\bar u_H,(1-\Cor)v_H) = 
F((1-\Cor)v_H)
\quad\text{for all }v_H\in V_H.
\end{equation}
By \eqref{e:cororth}, it is clear that this is equivalent to
\begin{equation}\label{e:msprimalPG}
a(\bar u_H,(1-\Cor)v_H) = 
F((1-\Cor)v_H)
\quad\text{for all }v_H\in V_H.
\end{equation}
A variant of this multiscale method
employs a problem-independent right-hand side and seeks
$u_H\in V_H$ such that
\begin{equation*}
a((1-\Cor)u_H,(1-\Cor)v_H)= 
F(v_H)
\quad\text{for all }v_H\in V_H
\end{equation*}
or, equivalently,
\begin{equation}\label{e:mspg}
a( u_H, (1-\Cor) v_H)
= F(v_H)
\quad\text{for all } v_H\in V_H.
\end{equation}

\subsection{Localization of the corrector problems}

Here, we briefly describe the localization technique of
\cite{MalqvistPeterseim2014}.
It was shown in \cite{MalqvistPeterseim2014} and
\cite[Lemma~4.9]{HenningPeterseim2013} that the method
is localizable in the sense that any $T\in\T_H$ and
any $j\in\{1,\dots,d\}$ satisfy
\begin{equation}\label{e:expdecay}
\|\nabla q_{T,j}\|_{L^2(\Omega\setminus\nei^m(T))}
\lesssim
\exp(-cm)
\|e_j\|_{L^2(T)} , \qquad m\geq 0.
\end{equation}
The exponential decay from \eqref{e:expdecay} suggests to
localize the computation \eqref{e:qTelementcorr} of the corrector
belonging to an element $T\in\T_H$ to a smaller domain, namely the extended element patch
$\Omega_T:=\nei^\ell(T)$ of order $\ell$.
The nonnegative integer $\ell$ is referred to as the
\emph{oversampling parameter}.
Let $W_{\Omega_T}\subseteq W$ denote the space of functions from $W$
that vanish outside $\Omega_T$.
On the patch,
in analogy to \eqref{e:qTelementcorr}, for any $v_H\in V_H$,
any $T\in\T_H$ and any $j\in\{1,\dots,d\}$,
the function $q_{T,j}\loc\in W_{\Omega_T}$ solves
\begin{equation}\label{e:qTelementcorrELL}
\int_{\Omega_T} \nabla w\cdot(A \nabla q_{T,j}\loc)\,dx
=
\int_{T} \nabla w \cdot(A e_j)\,dx
\quad\text{for all }w\in W_{\Omega_T}.
\end{equation}
Given $v_H\in V_H$, we define the corrector $\Cor\loc v_H\in W$
by
\begin{equation}\label{e:corexpansionELL}
\Cor\loc  v_H
= \sum_{T\in\T_H} \sum_{j=1}^d (\partial_j v_H|_T)  q_{T,j}\loc.
\end{equation}
A practical variant of \eqref{e:mspg}
is to seek $u_H\loc \in V_H$ such that
\begin{equation}\label{e:mspgELL}
a(u_H\loc, (1-\Cor\loc)v_H) = 
F(v_H)
\quad\text{for all }v_H\in V_H.
\end{equation}
This procedure is indispensable for actual computations
and the effect of the truncation of the domain on the error of
the multiscale method was analyzed in 
\cite{MalqvistPeterseim2014} and
\cite{HenningPeterseim2013}.
We will provide the error analysis for the method 
\eqref{e:mspgELL} in Subsection~\ref{ss:analysisquasilocal} below.

\subsection{Definition of the quasi-local effective coefficient}

In this subsection, we do not make any specific choice
for the oversampling parameter $\ell$.
In particular, the analysis covers the case that all element
patches $\Omega_T$ equal the whole domain $\Omega$.
We denote the latter case formally by $\ell=\infty$.

We re-interpret the left-hand side of \eqref{e:mspgELL} as a
non-local operator acting on standard finite element functions.
To this end, consider any $u_H,v_H \in V_H$.
We have
\begin{equation*}
a(u_H, (1-\Cor\loc)v_H) =
\int_\Omega \nabla u_H\cdot (A\nabla v_H)\,dx
-
\int_\Omega \nabla u_H\cdot (A \Cor\loc \nabla v_H)\,dx.
\end{equation*}
The second term can be expanded with 
\eqref{e:corexpansionELL} 
as
\begin{equation*}
\begin{aligned}
&
\int_\Omega \nabla u_H\cdot (A\nabla \Cor\loc v_H)\,dx
\\
&\qquad
=
 \sum_{T\in\T_H} \sum_{k=1}^d
  (\partial_k v_H|_T) 
 \int_\Omega \nabla u_H\cdot (A \nabla q_{T,k}\loc)\,dx
\\
 &\qquad
=
\sum_{K,T\in\T_H} 
\int_K \nabla u_H\cdot 
\left(\sum_{k=1}^d \fint_K(A(y) \nabla q_{T,k}\loc(y))\,dy  \;(\partial_k v_H|_T) \right)\,dx
\\
&\qquad
=
\sum_{K,T\in\T_H} \lvert K \rvert \,\lvert T\rvert \;\nabla u_H|_K 
  \cdot (\mathcal{K}_{T,K}\nabla v_H|_T)
\end{aligned}
\end{equation*}
for the matrix $\mathcal{K}_{T,K}\loc$ defined for any $K,T\in\T_H$ by
$$
(\mathcal{K}_{T,K}\loc)_{j,k} 
 := \frac{1}{|T|\,|K|} e_j\cdot\int_K A\nabla q_{T,k}\loc\,dx .
$$
Define the piecewise constant matrix field over $\T_H\times\T_H$,
for $T,K\in \T_H$ by
$$
\mathcal A_H\loc|_{T,K} 
:= \frac{\delta_{T,K}}{|K|} \fint_T A\,dx - \mathcal{K}_{T,K}\loc
$$
(where $\delta$ is the Kronecker symbol)
and the bilinear form $\mathfrak{a}\loc$ on $V_H\times V_H$ by
$$
\mathfrak{a}\loc(v_H,z_H):=
\int_\Omega\int_\Omega 
   \nabla v_H(y) \cdot ( \mathcal A_H\loc(x,y)\nabla z_H(x))\,dy\,dx 
   \quad\text{for any } v_H,z_H\in V_H.
$$
We obtain for all $v_H,z_H\in V_H$ that
\begin{equation}\label{e:frakaeq}
a(v_H, (1-\Cor\loc) z_H)
=
\mathfrak{a}\loc(v_H,z_H).
\end{equation}

\begin{remark}[notation]
For simplices $T,K\in\T_H$
with $x\in T$ and $y\in K$,
we will sometimes write
$\mathcal{K}\loc(x,y)$ instead of
$\mathcal{K}_{T,K}\loc$ (with analogous notation for 
$\mathcal{A}\loc$).
\end{remark}

Next, we state the equivalence of two 
multiscale formulations.

\begin{proposition}
A function $u_H\loc\in V_H$ solves \eqref{e:mspgELL} if and only
if it solves
\begin{equation}\label{e:msnonlocal}
\mathfrak{a}\loc(u_H\loc,v_H) = F(v_H) .
\end{equation}
\end{proposition}
\begin{proof}
This follows directly from the representation \eqref{e:frakaeq}.
\end{proof}

\begin{remark}
For $d=1$ and $I_H$ the standard nodal interpolation operator,
the corrector problems localize to one element and the presented
multiscale approach coincides with various known methods
(homogenization, MSFEM). The resulting effective coefficient
$\mathcal A_H\loc$ is diagonal and, thus, local.
This is no longer the case for $d\geq 2$.
\end{remark}

\subsection{Error analysis}\label{ss:analysisquasilocal}

This subsection presents an error estimate for the error produced
by the method \eqref{e:mspgELL} (and so by the method \eqref{e:msnonlocal}).
We begin by briefly summarizing
some results from \cite{MalqvistPeterseim2014}.
\begin{lemma}\label{l:mp14lemma}
Let $u\in V$ solve \eqref{e:modelweak} and
$\bar u_H\in V_H$ solve \eqref{e:msprimal}. Then we have 
the following properties.
\begin{itemize}
\item[(i)] $\bar u_H$ coincides with the quasi-interpolation of $u$,
           i.e., $\bar u_H = I_H u$.
\item[(ii)] The Galerkin orthogonality
              $a(u-(1-\Cor)I_H u, (1-\Cor) v_H) = 0$
              for all $v_H\in V_H$ is satisfied.
\item[(iii)] The error satisfies
              $\|\nabla (u-(1-\Cor)\bar u_H)\|_{L^2(\Omega)}
              \lesssim H \|f\|_{L^2(\Omega)}$.
\end{itemize}
\end{lemma}
\begin{proof}
See \cite{MalqvistPeterseim2014} for proofs.
\end{proof}

We define the following worst-case best-approximation
error
\begin{equation}\label{e:wcbadef}
\wcba(A,\T_H) 
:=
\sup_{g\in L^2(\Omega)\setminus\{0\}}
\inf_{v_H\in V_H} \frac{\|u(g)- v_H\|_{L^2(\Omega)}}{\|g\|_{L^2(\Omega)}} 
\end{equation}
where for $g\in L^2(\Omega)$, $u(g)\in V$ solves 
\eqref{e:modelweak} with right-hand side $g$.
Standard interpolation and stability estimates show that always
$\wcba(A,\T_H)\lesssim H$, but it may behave better in certain regimes. 
E.g., in a periodic homogenization problem with some small parameter
$\varepsilon$ and some smooth homogenized solution $u_0\in H^2(\Omega)$, 
the best approximation error is dominated by the best approximation error
of $u_0$ in the regime $H\lesssim\sqrt{\varepsilon}$ where it scales 
like $H^2$. By contrast, the error is typically not improved in the 
regime $\sqrt{\varepsilon}\gtrsim H\gtrsim \varepsilon$. 
This non-linear behavior of the best-approximation error in the 
pre-asymptotic regime is prototypical for homogenization problems 
with scale separation
and explains why the rough bound $H$ is suboptimal. 

The following result states an $L^2$ error estimate for
the method \eqref{e:mspg}.
The result is surprising because the perturbation of the right-hand 
side seems to be of order $H$ at first glance. In cases of scale 
separation the quadratic rate is indeed observed in the regime 
$H\lesssim \sqrt{\varepsilon}$ and cannot be explained by naive 
estimates.

\begin{proposition}\label{p:mspgapriori}
The solutions $u\in V$ to \eqref{e:modelweak}
and $u_H\in V_H$ to \eqref{e:mspg}
for right-hand side $f\in L^2(\Omega)$ satisfy
the following error estimate
\begin{equation*}
\frac{\|u - u_H\|_{L^2(\Omega)}}{\|f\|_{L^2(\Omega)}}
\lesssim
H^2 + \wcba(A,\T_H) .
\end{equation*}
\end{proposition}
\begin{proof}
Let $f\in L^2(\Omega)\setminus\{0\}$ and let $\bar u_H\in V_H$ solve 
\eqref{e:msprimalPG}.
We begin by analyzing the error
$e_H:= u_H-\bar u_H$. Let $z\in V$ denote the solution to
\begin{equation*}
a(v,z) = (e_H,I_H v)_{L^2(\Omega)}
\quad\text{for all }v\in V.
\end{equation*}
To see that the right-hand side is indeed represented by
an $L^2$ function, note that $I_H$ is continuous on $L^2(\Omega)$
and, hence, the right-hand side has a Riesz representative
$\tilde e\in L^2(\Omega)$ such that
$(e_H,I_H v)_{L^2(\Omega)} = (\tilde e,v)_{L^2(\Omega)}$.
In particular, $z$ solves \eqref{e:modelweak} with
right-hand side $\tilde e$.
Its $L^2$ norm is bounded with \eqref{e:IHstabL2} as follows
\begin{equation*}
\|\tilde e\|_{L^2(\Omega)}^2
= (e_H,I_H \tilde e_H)_{L^2(\Omega)}
\lesssim \|e_H\|_{L^2(\Omega)}
  \|\tilde e\|_{L^2(\Omega)},
\end{equation*}
hence
\begin{equation}\label{e:eHtilde}
\|\tilde e\|_{L^2(\Omega)}\lesssim \|e_H\|_{L^2(\Omega)}.
\end{equation}
We note that, for any $w\in W$,
we have
$a(w,z)=(e_H,I_Hw)_{L^2(\Omega)}=0$. 
Thus, we have
$a(e_H,\Cor z)=a(\Cor e_H,z)=0$.
With $(1-\Cor)z=(1-\Cor)I_H z$ we conclude
\begin{equation}\label{e:eHduality}
\|e_H\|_{L^2(\Omega)}^2
=
a(e_H ,z)
=
a(e_H,(1-\Cor)I_H z).
\end{equation}
Elementary algebraic manipulations with the projection $I_H$ show 
that
$$
- \Cor I_H z = (1-I_H)\big((1-\Cor)I_Hz-z\big)+(1-I_H)z.
$$
The relation \eqref{e:eHduality} and 
the solution properties \eqref{e:msprimalPG} and \eqref{e:mspg},
thus, lead to
\begin{equation}\label{e:Fsplit}
\|e_H\|_{L^2(\Omega)}^2
=
F(\Cor I_H z) 
=
\lvert
F ((1-I_H)((1-\Cor)I_Hz-z))
+
F((1-I_H)z)
\rvert
.
\end{equation}
We proceed by estimating the two terms on the right-hand side
of \eqref{e:Fsplit} separately.
For the second term in \eqref{e:Fsplit},
the $L^2$-best approximation property of $I_H$ 
and \eqref{e:eHtilde} reveal
\begin{equation}\label{e:dualIHv}
\begin{aligned}
\lvert F((1-I_H)z) \rvert
&\lesssim
\|f\|_{L^2(\Omega)} \|\tilde e\|_{L^2(\Omega)}
\inf_{v_H\in V_H} \frac{\|z- v_H\|_{L^2(\Omega)}}{\|\tilde e\|_{L^2(\Omega)}}
\\
&
\lesssim
\|f\|_{L^2(\Omega)} \|e_H\|_{L^2(\Omega)}
\wcba(A,\T_H) .
\end{aligned}
\end{equation}
For the first term in \eqref{e:Fsplit},
we obtain with the stability of $I_H$
and the Cauchy inequality that
\begin{equation*}
\lvert F ((1-I_H)((1-\Cor)I_Hz-z)) \rvert
\lesssim
\|f\|_{L^2(\Omega)}
\|z- (1-\Cor)I_Hz\|_{L^2(\Omega)}.
\end{equation*}
Let $\tilde g:=z- (1-\Cor)I_Hz$
and let $\zeta\in V$ denote the solution to
\begin{equation*}
a(\zeta,v) = (\tilde g, v)_{L^2(\Omega)}
\quad\text{for all }v\in V.
\end{equation*}
As stated in
Lemma~\ref{l:mp14lemma}(i),
the function $I_Hz\in V_H$ is the Galerkin approximation
to $z$ with method \eqref{e:msprimal}
with right-hand side $\tilde e$.
We, thus, have by symmetry of $a$ and the Galerkin orthogonality 
from Lemma~\ref{l:mp14lemma}(ii)
that
\begin{equation*}
\begin{aligned}
\|z- (1-\Cor)I_Hz\|_{L^2(\Omega)}^2
&=
a(\zeta,z- (1-\Cor)I_Hz)
\\
&=
a(\zeta-(1-\Cor)I_H\zeta,z- (1-\Cor)I_Hz).
\end{aligned}
\end{equation*}
Continuity of $a$ and Lemma~\ref{l:mp14lemma}(iii)
reveal that this is bounded by
$$
  H^2 \|\tilde{g}\|_{L^2(\Omega)}  \|\tilde e\|_{L^2(\Omega)}
  =
  H^2 \|z- (1-\Cor)I_Hz\|_{L^2(\Omega)}  \|\tilde e\|_{L^2(\Omega)}.
$$
Altogether, with \eqref{e:Fsplit},
\begin{equation*}
\frac{\|e_H\|_{L^2(\Omega)}}{\|f\|_{L^2(\Omega)}}
\lesssim
H^2 + \wcba(A,\T_H) .
\end{equation*}
Since
\begin{equation*}
\frac{\|u-\bar u_H\|_{L^2(\Omega)}}{\|f\|_{L^2(\Omega)}}
\lesssim
\wcba(A,\T_H),
\end{equation*}
(which follows from the fact that $\bar u_H = I_H u$),
the triangle inequality concludes the proof.
\end{proof}

With similar arguments it is possible to prove
that the coupling
$\ell\approx \lvert\log H\rvert$ is sufficient to derive 
the error bound
\begin{equation}\label{e:mspgl2est}
\| u- u_H\loc \|_{L^2(\Omega)}
\lesssim (H^2+\wcba(A,\T_H) )\, \|f\|_{L^2(\Omega)}.
\end{equation}
The proof is based on a similar argument as
in Proposition~\ref{p:mspgapriori}:
Since the $L^2$ distance of $u-\bar u_H\loc$ is controlled
by the right-hand side of \eqref{e:mspgl2est}
\cite{HenningPeterseim2013}
where $\bar u_H\loc$ solves a modified version of 
\eqref{e:mspgELL} with right-hand side
$F((1-\Cor\loc)v_H)$,
it is sufficient to control
$u_H\loc-\bar u_H\loc$ in the $L^2$ norm. This can be done with
a duality argument similar to that from the proof of 
Proposition~\ref{p:mspgapriori}. The additional tool needed therein
is the fact that
$$
\| \nabla(\Cor-\Cor\loc) I_H z\|_{L^2(\Omega)}
\lesssim \exp(-c\ell) C(\ell) \|\nabla z\|_{L^2(\Omega)}
$$
for the dual solution $z$
(see \cite[Proof of Thm.~4.13]{HenningPeterseim2013} for 
an outline of a proof)
where $C(\ell)$ is an overlap constant depending polynomially
on $\ell$.
The choice of $\ell\approx\lvert\log H\rvert$ therefore leads
to \eqref{e:mspgl2est}. The details are omitted here 
and the reader is referred to
\cite{MalqvistPeterseim2014,HenningPeterseim2013,Peterseim:2015,Kornhuber.Peterseim.Yserentant:2016}.

\section{Local effective coefficient}\label{s:local}
Throughout this section we consider oversampling parameters
chosen as $\ell\approx\lvert \log H\rvert$.

\subsection{Definition of the local effective coefficient}
The exponential decay
motivates to approximate the non-local bilinear form
$\mathfrak{a}\loc(\cdot,\cdot)$ by a quadrature-like procedure:
Define the piecewise constant coefficient
$A_H\loc\in P_0(\T_H;\mathbb R^{d\times d})$ by
\begin{equation*}
A_H\loc|_T:=
\fint_T A \,dx-\sum_{K\in\T_H} |K|\, \mathcal{K}_{T,K}\loc.
\end{equation*}
and the bilinear form $\tilde a\loc$ on $V\times V$ by
$$
\tilde a\loc(u,v) : = \int_\Omega \nabla u \cdot ( A_H\loc \nabla v)\,dx.
$$

\begin{remark}\label{r:homo1}
In analogy to classical periodic homogenization,
the local effective coefficient $A_H\loc$ can be written as
\begin{equation*}
\begin{aligned}
(A_H\loc)_{j,k}|_T
&
=
|T|^{-1} \int_{\Omega_T} e_j\cdot
                       \big(   A (\chi_T e_k-\nabla q_{T,k}\loc) \big)
\\
&
=
|T|^{-1} \int_{\Omega_T} (e_j-\nabla q_{T,j}\loc)\cdot
                       \big(   A (\chi_T e_k-\nabla q_{T,k}\loc) \big)
\end{aligned}
\end{equation*}
for the characteristic function $\chi_T$ of $T$
and the slightly enlarged averaging domain $\Omega_T$.
See Section~\ref{s:periodic} for further analogies to
homogenization theory in the periodic case.
\end{remark}

The localized multiscale method is to seek
$\uHt\loc\in V_H$ such that
\begin{equation}\label{e:mslump}
\tilde a\loc(\uHt\loc,v_H) = F(v_H)
\quad\text{for all } v_H\in V_H.
\end{equation}
The unique solvability of \eqref{e:mslump} is not guaranteed
a~priori.
It must be checked a~posteriori whether positive spectral
bounds $\alpha_H$, $\beta_H$ on $A_H\loc$ exist in the sense
of \eqref{e:Abounds}.
Throughout this paper we assume that such bounds exist,
that is, we assume that there exist positive numbers
$\alpha_H$, $\beta_H$ such that
\begin{equation}\label{e:AHbounds}
\alpha_H|\xi|^2
\leq \xi \cdot ( A_H\loc(x)\xi)
\leq \beta_H |\xi|^2
\end{equation}
for all $\xi\in\mathbb{R}^d$ and almost all $x\in\Omega$.

\subsection{Error analysis}
The goal of this section is to 
establish an error estimate for the error 
$$
\| u - \uHt\loc\|_{L^2(\Omega)}.
$$
Let $u_H\loc\in V_H$ solve
\eqref{e:mspgELL}. Then
the error estimate
\eqref{e:mspgl2est}
leads to
the a~priori error estimate
\begin{equation}\label{e:msHfestimate}
\| u- u_H\loc\|_{L^2(\Omega)} 
\lesssim 
(H^2+\wcba(A,\T_H) ) \, \|f\|_{L^2(\Omega)}.
\end{equation}
We employ the triangle inequality and merely estimate the 
difference $\| u_H\loc - \uHt\loc \|_{L^2(\Omega)}$.

With the finite localization parameter $\ell$, the
quasi-local coefficient $\mathcal A\loc$ is sparse
in the sense that $\mathcal A\loc(x,y) = 0$ whenever
$|x-y|> C \ell H$.
We note the following lemma which will be employed in the error
analysis.

\begin{lemma}\label{l:LpboundELL}
Given some $x\in\Omega$ with $x\in T$ for some $T\in\T_H$,
we have
\begin{equation*}
\| \mathcal K\loc(x,y)\|_{L^2(\Omega,dy)}
\lesssim
H^{-d/2} .
\end{equation*}
\end{lemma}
\begin{proof}
From the definition of $\mathcal K\loc$, the boundedness of
$A$ and the H\"older inequality
we obtain for any $j,k\in\{1,\dots,d\}$ that
\begin{equation*}
| (\mathcal{K}\loc_{T,K})_{j,k} | \lesssim
\frac{1}{|T|\,|K|} \|\nabla q_{T,k}\|_{L^1(K)}
\lesssim
\frac{1}{|T|\,|K|^{1/2}} \|\nabla q_{T,k}\|_{L^2(K)}.
\end{equation*}
Hence, we conclude with the stability of problem 
\eqref{e:qTelementcorrELL} and  $\|e_k\|_{L^2(T)}^2  = |T|$ that
\begin{equation*} 
\| \mathcal K\loc(x,y)\|_{L^2(\Omega,dy)}^2
=
\sum_{K\in\T_H} |K| |\mathcal K\loc_{T,K}|^2
=
|T|^{-2} \|\nabla q_{T,k}\|_{L^2(\Omega)}^2
\lesssim H^{-d}.
\end{equation*}
This implies the assertion.
\end{proof}

In what follows, we abbreviate
\begin{equation}\label{e:rhodef}
\rho:=C H \lvert\log H\rvert
\end{equation}
for some appropriately chosen constant $C$.

\begin{proposition}[error estimate I]\label{p:posteriori}
Assume that \eqref{e:AHbounds} is satisfied.
Let $u_H\loc\in V_H$ solve \eqref{e:msnonlocal}
and let $\uHt\loc$ solve \eqref{e:mslump}.
Then,
\begin{equation*}
\|\nabla( u_H\loc - \uHt\loc)\|_{L^2(\Omega)}
\lesssim
H^{-d/2}
     \Big\|
       \| \nabla \uHt\loc(y) - \nabla \uHt\loc(x)\|_{L^2(B_{\rho}(x),dy)}
     \Big\|_{L^2(\Omega,dx)} .
\end{equation*}

\end{proposition}

\begin{proof}

Denote $e_H:= \uHt\loc - u_H\loc$.
In the idealized case, $\ell=\infty$, the orthogonality
\eqref{e:cororth} and relation \eqref{e:frakaeq} show that
$$
\|\nabla(1-\Cor\loc)e_H\|_{L^2(\Omega)}^2
\lesssim
\mathfrak{a}\loc(e_H,e_H) .
$$
The case $\ell\gtrsim|\log H|$
again follows ideas from \cite{MalqvistPeterseim2014}
with the exponential-in-$\ell$ closeness of $\Cor$ and $\Cor_\ell$
and is merely sketched here.
From the stability of $I_H$ and the properties of the 
fine-scale projection $\Cor\loc$
we observe (with contrast-dependent constants)
\begin{equation*}
\begin{aligned}
\|\nabla e_H\|_{L^2(\Omega)}^2
=
\|\nabla I_H e_H\|_{L^2(\Omega)}^2
&
=
\|\nabla I_H (1-\Cor\loc)e_H\|_{L^2(\Omega)}^2
\\
&
\lesssim
\|\nabla(1-\Cor\loc)e_H\|_{L^2(\Omega)}^2
\\
&
\lesssim
\mathfrak{a}\loc(e_H,e_H)
+ \exp(-c\ell) \|\nabla e_H\|_{L^2(\Omega)}^2
\end{aligned}
\end{equation*}
for some constant $c>0$.
Hence, with positive constants $C_1$, $C_2$,
\begin{equation*}
\|\nabla e_H\|_{L^2(\Omega)}^2
\leq
C_1
\mathfrak{a}\loc(e_H,e_H)
+ C_2 \exp(-c\ell) \|\nabla e_H\|_{L^2(\Omega)}^2.
\end{equation*}
If, for some sufficiently large $r$, the parameter $\ell$ is chosen
to satisfy $\ell\geq r\lvert\log H\rvert$ such that
$ C_2 \exp(-c\ell) \leq 1/2$,
then the second term on the right-hand side
can be absorbed.
Thus,
we proceed with \eqref{e:msnonlocal} and \eqref{e:mslump} as
\begin{equation*}
\|\nabla e_H\|_{L^2(\Omega)}^2
\lesssim
\mathfrak{a}\loc(\uHt\loc - u_H\loc,e_H)
=
\mathfrak{a}\loc (\uHt\loc,e_H)
-
\tilde a\loc(\uHt\loc,e_H) .
\end{equation*}
The right-hand side can be rewritten as
\begin{equation*}
\begin{aligned}
&
\mathfrak a\loc (\uHt\loc,e_H)
-
\tilde a\loc(\uHt\loc,e_H) 
\\
&
=
\int_\Omega 
 \int_\Omega
   (\nabla \uHt\loc(y)-\nabla \uHt\loc(x))
   \cdot
   \bigg[ \mathcal A_H\loc(x,y) 
            \nabla e_H(x)\bigg]
         \,dy\,dx
\\
&\quad
+
\int_\Omega 
   \bigg[\uHt\loc(x)\cdot
      \left(\int_\Omega \mathcal A_H\loc(x,y)\,dy - A_H\loc(x)\right)
      \nabla e_H(x)
     \bigg]
          \,dx.
\end{aligned}
\end{equation*}
The second term vanishes by definition of $A_H\loc$.
Hence, the combination of the preceding arguments with
the Cauchy inequality leads to
\begin{equation*}
\|\nabla e_H\|_{L^2(\Omega)}^2\lesssim
\|\nabla e_H\|_{L^2(\Omega)}
\Big\| \int_{B_\rho(x)} \mathcal A_H\loc(x,y) ^*
          (\nabla \uHt\loc(y)-\nabla \uHt\loc(x)) \,dy \Big\|_{L^2(\Omega,dx)} 
,
\end{equation*}
where it was used that $\mathcal A_H\loc(x,y) =0$ whenever
$|x-y|>\rho$.
Note that $(\nabla \uHt\loc(y)-\nabla \uHt\loc(x))=0$
for all $x$ and $y$ that belong to the same element $T\in\T_H$.
Thus, $\mathcal A_H\loc(x,y)$ in the above expression can be
replaced by $\mathcal K_H\loc(x,y)$.
This and 
division by $\|\nabla e_H\|_{L^2(\Omega)}$  lead to
\begin{equation}\label{e:aposteriori}
\begin{aligned}
&
\|\nabla e_H\|_{L^2(\Omega)}
\\
&
\qquad
\lesssim
\sqrt{\phantom{.}}
\left(\int_\Omega 
  \bigg|
  \int_{B_{\rho}(x)}
     \mathcal K\loc(x,y)^* (\nabla \uHt\loc(y) - \nabla \uHt\loc(x))
   \,dy
  \bigg|^2 dx
\right)  .
\end{aligned}
\end{equation}
This term
can be bounded with the 
Cauchy inequality
and Lemma~\ref{l:LpboundELL} by
\begin{equation*}
\begin{aligned}
&\quad
\sqrt{\phantom{.}}
\left(\int_\Omega 
  \bigg|
    \|\mathcal K\loc(x,y)\|_{L^2(B_{\rho}(x),dy)}  
     \| \nabla \uHt\loc(y) - \nabla \uHt\loc(x)\|_{L^2(B_{\rho}(x),dy)}
  \bigg|^2 \,dx
\right)
\\
&
\qquad
\lesssim
 H^{-d/2} 
     \Big\|
       \| \nabla \uHt\loc(y) - \nabla \uHt\loc(x)\|_{L^2(B_{\rho}(x),dy)}
     \Big\|_{L^2(\Omega,dx)}.
\end{aligned}
\end{equation*}
This finishes the proof.
\end{proof}
It is worth noting that the error bound in 
Proposition~\ref{p:posteriori} can be evaluated without knowledge
of the exact solution. Hence, Proposition~\ref{p:posteriori}
can be regarded as an a~posteriori error estimate.
Formula \eqref{e:aposteriori} could also be an option if it is
available.
We expect Proposition~\ref{p:posteriori} to be rather sharp.
Below we provide the main a~priori error estimate,
Proposition~\ref{p:priori}, 
which is fundamental for the mentioned link between 
analytical and numerical homogenization.
The following technical lemma is required.

\begin{lemma}[existence of a regularized coefficient]\label{l:smoothed}
Let $A_H\in P_0(\T_H;\mathbb R^{d\times d})$ be a piecewise
constant field of $d\times d$ matrices that satisfies
the spectral bounds \eqref{e:AHbounds}.
Then there exists a Lipschitz continuous coefficient
$A_H\smooth \in W^{1,\infty}(\Omega;\mathbb R^{d\times d})$
satisfying the following three properties.
1) The piecewise integral mean is conserved, i.e.,
\begin{equation*}
\int_T A_H\smooth \,dx = \int_T A_H\,dx
\quad\text{for all }T\in \T_H.
\end{equation*}
2) The eigenvalues of $\sym(A_H\smooth)$ lie in the interval
$[\alpha_H/2,2\beta_H]$.
3) The derivative satisfies the bound
\begin{equation*}
\|\nabla A_H\smooth\|_{L^\infty(\Omega)}
\leq C
\eta(A_H)
\end{equation*}
for some constant $C$ that depends on the shape-regularity of
$\T_H$
and for the expression
\begin{equation}\label{e:etadef}
\eta(A_H) :=
H^{-1} \| [A_H] \|_{L^\infty(\mathcal F_H)}
\big(1+\alpha_H^{-1}\| [A_H] \|_{L^\infty(\mathcal F_H)}\big).
\end{equation}
Here $[\cdot]$ defines the inter-element jump and
$\mathcal F_H$ denotes the set of
interior hyper-faces of $\T_H$.
\end{lemma}

\begin{proof}
Consider a refined triangulation $\T_L$
resulting from $L$ uniform refinements of $\T_H$.
In particular, the mesh-size in $\T_L$ is of the order
$2^{-L} H$.
Let $E_L A_H$ denote
the $\T_L$-piecewise affine and continuous function
that takes at every interior vertex
the arithmetic mean of the nodal values of $A_H$ on the 
adjacent elements of $\T_L$ (similar to \eqref{e:EHdef}).
Clearly, for this convex combination
the eigenvalues of $\sym(E_L A_H)$ range within the interval
$[\alpha_H,\beta_H]$.
It is not difficult to prove that, for any $T\in\T_H$,
\begin{equation}\label{e:AHELbound}
\fint_T | A_H - E_L A_H | \,dx
\lesssim
2^{-L} \| [A_H] \|_{L^\infty(\mathcal F_H(\omega_T))}
\end{equation}
as well as
\begin{equation}\label{e:AHELsup}
\| A_H - E_L A_H \|_{L^\infty(T)}
\lesssim
\| [A_H] \|_{L^\infty(\mathcal F_H(\omega_T))} .
\end{equation}
Here,
$\mathcal F_H(\omega_T)$ denotes the set of
interior hyper-faces of $\T_H$
that share a point with $T$.
Let, for any $T\in\T_H$, $b_T\in H^1_0(T)$ denote a 
positive polynomial bubble function with
$\fint_T b_T\,dx =1$ and $\|b_T\|_{L^\infty(T)}\approx 1$.
The regularized coefficient
$A_H\smooth= E_L(A_H) + 
 b_T \fint_T (A_H - E_L(A_H))\,dx$ 
has, for any $T\in\T_H$, the integral mean
$\fint_T A_H\smooth\,dx=\fint_T A_H\,dx$.
For any $\xi\in\mathbb R^d$ with $|\xi|=1$ and
any $T\in\T_H$, the estimate \eqref{e:AHELbound} shows
\begin{equation*}
\begin{aligned}
\left|\xi\cdot\fint_T (A_H - E_L A_H)\,dx\;b_T \xi\right|
&\leq
 \left|\fint_T (A_H - E_L A_H)\,dx\;b_T\right| 
\\
&\leq
  C 2^{-L} \| [A_H] \|_{L^\infty(\mathcal F_H(\omega_T))}.
\end{aligned}
\end{equation*}
If $L$ is chosen to be of the order
$\lvert\log(\alpha_H^{-1}C\| [A_H] \|_{L^\infty(\mathcal F_H)})\rvert$
(for small jumps of $ A_H$ it can be chosen of order $1$),
then
\begin{equation*}
\left|\xi\cdot\fint_T (A_H - E_L A_H)\,dx\;b_T \xi\right|
\leq
\alpha/2.
\end{equation*}
This and the triangle inequality prove the claimed spectral bound
on $\sym(A_H\smooth)$.
For the bound on the derivative of $A_H\smooth$,
let $t\in T_L$ and $T\in\T_H$ such that $t\subseteq T$.
The diameter of $t$ is of order 
$2^{-L}H$.
Since $\|\nabla b_T\|_{L^\infty(T)}\lesssim H^{-1}$,
the triangle and inverse inequalities
therefore yield with the above choice of $L$
(note that $\nabla (A_H|_T) = 0$)
\begin{equation*}
\begin{aligned}
\| \nabla A_H\smooth \|_{L^\infty(t)}
&
\lesssim \| \nabla (A_H - E_L(A_H))\|_{L^\infty(t)}
 + H^{-1} \| A_H - E_L(A_H)\|_{L^\infty(t)}
\\
&
\lesssim
H^{-1} \| [A_H] \|_{L^\infty(\mathcal F_H(\omega_T))}
\big(1+\alpha_H^{-1}\| [A_H] \|_{L^\infty(\mathcal F_H(\omega_T)}\big).
\end{aligned}
\end{equation*}
This proves the assertion.
\end{proof}

By Lemma~\ref{l:smoothed}, there
exists a  coefficient
$A_H\smooth\in W^{1,\infty}(\Omega)$
such that
$A_H\loc$ is the piecewise $L^2$ projection of $A_H\smooth$
onto the piecewise constants.
Let $u\smooth\in V$ solve
\begin{equation}\label{e:PDEsmooth}
\int_\Omega\nabla {u\smooth}\cdot ( A_H\smooth \nabla v) \,dx =
 F(v)\quad\text{for all } v\in V.
\end{equation}
In particular, $\uHt$ is the finite element approximation to
$u\smooth$.
In the following, $s$ refers to the 
$H^{1+s}(\Omega)$ regularity
index of a function.
Recall that the $H^{1+s}(\Omega)$ norm
\cite{Adams1975}
of some a function $v$ is given by
\begin{equation}\label{e:W1snormdef}
\|v\|_{H^{1+s}(\Omega)}
=
\left[
\|v\|_{H^{1}(\Omega)}^2
+
\int_\Omega  
     \int_\Omega
     \frac{| \nabla v(x) - \nabla v(y)|^2}{\lvert x-y\rvert^{d+2s}}
      \,dy\,dx  
\right]^{1/2} .
\end{equation}
We have the following error estimate.

\begin{proposition}[error estimate II]\label{p:priori}
Let $\ell\approx\lvert\log H\rvert$ and assume
that \eqref{e:AHbounds} is satisfied.
Let $u_H\loc$ solve \eqref{e:mspgELL}
and let $\uHt\loc$ solve \eqref{e:mslump}.
Assume furthermore that the solution $u\smooth$ to
\eqref{e:PDEsmooth} belongs to $H^{1+s}(\Omega)$
for some $0<s\leq 1$.
Then,
\begin{equation*}
\|\nabla( u_H\loc - \uHt\loc)\|_{L^2(\Omega)}
\lesssim
H^s
\lvert \log H\rvert^{s+d/2} 
\big(1+\eta(A_H\loc)\big)^s
\| f \|_{L^2(\Omega)}.
\end{equation*}

\end{proposition}
\begin{proof}
Recall the estimate from Proposition~\ref{p:posteriori}
\begin{equation*}
\|\nabla( u_H\loc - \uHt\loc)\|_{L^2(\Omega)}
\lesssim 
H^{-d/2}
     \Big\|
       \| \nabla \uHt\loc(y) - \nabla \uHt\loc(x)\|_{L^2(B_{\rho}(x),dy)}
     \Big\|_{L^2(\Omega,dx)} .
\end{equation*}
To bound the norm on the right-hand side,
we denote $e:=\nabla(\uHt\loc - u\smooth)$
and infer with the triangle inequality
\begin{equation}\label{e:triineq}
\begin{aligned}
&
     \Big\|
       \| \nabla \uHt\loc(y) - \nabla \uHt\loc(x)\|_{L^2(B_{\rho}(x),dy)}
     \Big\|_{L^2(\Omega,dx)} 
\\
&
\qquad\qquad
\leq
   \Big\|
       \| e(y)\|_{L^2(B_{\rho}(x),dy)}
     \Big\|_{L^2(\Omega,dx)} 
\\
&
\qquad\qquad\quad
+
   \Big\|
       \| \nabla u\smooth(y) - \nabla u\smooth(x)\|_{L^2(B_{\rho}(x),dy)}
     \Big\|_{L^2(\Omega,dx)} 
\\
&
\qquad\qquad\quad
+
   \Big\|
       \| e(x)\|_{L^2(B_{\rho}(x),dy)}
     \Big\|_{L^2(\Omega,dx)} .
\end{aligned}
\end{equation}
The square of the first term on the right-hand side of \eqref{e:triineq}
satisfies
\begin{equation*}
\begin{aligned}
\Big\|
       \| e(y)\|_{L^2(B_{\rho}(x),dy)}
     \Big\|_{L^2(\Omega,dx)}^2
&=
\int_\Omega
\int_{B_\rho(x)}
|e(y)|^2
\,dy
\,dx
\\
&
=
\int_\Omega
\int_{\{x \text{ with } y\in B_\rho(x)\}}
|e(y)|^2
\,dx
\,dy
\lesssim
\rho^d \| e\|_{L^2(\Omega)}^2.
\end{aligned}
\end{equation*}
Similarly, the third term on the right-hand side of
\eqref{e:triineq} satisfies
\begin{equation*}
\Big\|
       \| e(x)\|_{L^2(B_{\rho}(x),dy)}
     \Big\|_{L^2(\Omega,dx)}^2
=
\int_\Omega
\int_{B_\rho(x)}
|e(x)|^2
\,dy
\,dx
\lesssim
\rho^d \| e\|_{L^2(\Omega)}^2.
\end{equation*}
The second term on the right-hand side of
\eqref{e:triineq} reads
for any $0<s<1$ as
 \begin{equation*}
 \begin{aligned}
&   \Big\|
       \| \nabla u\smooth(y) - \nabla u\smooth(x)\|_{L^2(B_{\rho}(x),dy)}
     \Big\|_{L^2(\Omega,dx)}  
\\
&
\qquad
=
\rho^{(d+2s)/2}
\left(\int_\Omega 
      \int_{B_\rho(x)}
     \frac{| \nabla u\smooth(x) - \nabla u\smooth(y)|^2}{\rho^{d+2s}}
      \,dy\,dx
 \right)^{1/2} \bigg]
 \\
&
\qquad
\lesssim
\rho^{(d+2s)/2} 
\|u\smooth\|_{H^{1+s}(\Omega)} .
 \end{aligned}
 \end{equation*}
Here we have used the representation
\eqref{e:W1snormdef} and the fact that the value of the double
integral increases, when, first, in the denominator $\rho$ is replaced
by $\lvert x-y\rvert$ and thereafter the integration domain
of the inner integral is replaced by $\Omega$.
In conclusion,
\begin{equation*}
\begin{aligned}
&
     \Big\|
       \| \nabla \uHt\loc(y) - \nabla \uHt\loc(x)\|_{L^2(B_{\rho}(x),dy)}
     \Big\|_{L^2(\Omega,dx)} 
\\
&
\qquad\qquad
\lesssim
\rho^{d/2} \| e\|_{L^2(\Omega)}
+
\rho^{(d+2s)/2} 
\|u\smooth\|_{H^{1+s}(\Omega)}.
\end{aligned}
\end{equation*}
Since $\uHt\loc$ is the finite element approximation to
$\nabla u\smooth$, standard a~priori error estimates for the
Galerkin projection yield
\begin{equation*}
\| e\|_{L^2(\Omega)}
\lesssim
H^s \|u\smooth\|_{H^{1+s}(\Omega)}.
\end{equation*}
Thus,
\begin{equation}\label{e:estIIstep}
     \Big\|
       \| \nabla \uHt\loc(y) - \nabla \uHt\loc(x)\|_{L^2(B_{\rho}(x),dy)}
     \Big\|_{L^2(\Omega,dx)} 
\\
\lesssim
\rho^{(d+2s)/2} 
\|u\smooth\|_{H^{1+s}(\Omega)}.
\end{equation}
If $u\smooth$ belongs to $H^2(\Omega)$, then
the results of \cite{Grisvard1985,Dauge1988,Melenk2002}
lead to 
\begin{equation}\label{e:W2qstab}
\begin{aligned}
\|u\smooth\|_{H^2(\Omega)}
&\lesssim 
\|A_H\smooth\|_{W^{1,\infty}(\Omega)}
 (\| f \|_{L^2(\Omega)} + \|u\smooth\|_{H^1(\Omega)})
\\
&\lesssim
\|A_H\smooth\|_{W^{1,\infty}(\Omega)}
\| f \|_{L^2(\Omega)} .
\end{aligned}
\end{equation}
The assertion in $H^{1+s}(\Omega)$ can be proved with
an operator interpolation argument.
Indeed,
as shown in \cite{Grisvard1985}, the 
operator $-\ddiv (A\smooth\nabla \cdot)$ maps
$H^2(\Omega)\cap H^1_0(\Omega)$ to a closed subspace 
$Y_1$ of $L^2(\Omega)$.
Let $T$ denote the solution operator, which maps
$Y_1$ to $X_1:=H^2(\Omega)$ 
and furthermore maps
$Y_0:=L^2(\Omega)$ to
$X_0:=H^1(\Omega)$.
The real method of Banach space interpolation
\cite{BerghLoefstroem1976}
shows that
$H^{1+s}(\Omega) = [X_0,X_1]_{s,2}$,
which together with the $H^1$ stability of the problem and
\eqref{e:W2qstab} proves
$$
\|u\smooth\|_{H^{1+s}(\Omega)}
\lesssim
\|A_H\smooth\|_{W^{1,\infty}(\Omega)}^s
\| f \|_{L^2(\Omega)} .
$$
The combination with Lemma~\ref{l:smoothed}
proves
\begin{equation*}
\|u\smooth\|_{H^{1+s}(\Omega)}
\lesssim
 \big(1+\eta(A_H\loc)\big)^s \;
 \| f \|_{L^2(\Omega)}.
\end{equation*}
The combination with Proposition~\ref{p:posteriori}
and \eqref{e:estIIstep} proves
\begin{equation*}
\begin{aligned}
\|\nabla( u_H\loc - \uHt\loc)\|_{L^2(\Omega)}
&
\lesssim 
H^{-d/2}
\rho^{(d+2s)/2} 
\|u\smooth\|_{H^{1+s}(\Omega)}
\\
&
\lesssim
H^s
\lvert \log H\rvert^{s+d/2} 
 \big(1+\eta(A_H\loc)\big)^s \;
 \| f \|_{L^2(\Omega)}
.
\end{aligned}
\end{equation*}
This implies the assertion.
\end{proof}

\begin{remark}[homogenization indicator]\label{r:homcrit}
If the relations
\begin{equation*}
H^{-1}\| [A_H\loc] \|_{L^\infty(\mathcal F_H)} \lesssim 1
\quad\text{and}\quad
\alpha_H^{-1} H
\lesssim 1
\end{equation*}
are satisfied, then the multiplicative constant
in Proposition~\ref{p:priori} is of moderate size.
Hence, we interpret $\eta(A_H\loc)$ as a homogenization
indicator and the above relations
as a \emph{homogenization criterion}.
\end{remark}

\begin{remark}[local mesh-refinement]
We furthermore remark that local versions of 
$\eta(A_H\loc)$
involving the jump information
$H^{-1}\| [A_H\loc] \|_{L^\infty(F)}$ for interior interfaces
$F$ may be used as refinement indicators for local mesh-adaptation.
This possibility shall, however, not be further discussed here.
\end{remark}

\begin{remark}[global homogenized coefficient]
If the global variations of $A_H\loc$ are small in the sense that
there are positive constants $c_1$, $c_2$ such that,
almost everywhere,
$$
 c_1 |\xi|^2 
 \leq \xi\cdot (A_H\loc \xi)
  \leq c_2 |\xi|^2 \quad\text{for any } \xi\in \mathbb R^d
$$
holds with $|c_2-c_1|\lesssim H$,
then $A_H\loc$ can be replaced by $\fint_\Omega A_H\loc\,dx$ 
without effecting the accuracy.
\end{remark}

The combination of Proposition~\ref{p:priori}
with \eqref{e:msHfestimate} leads to the following a~priori
error estimate. The parameter $s$ therein is determined by the
elliptic regularity of the model problem with a $W^{1,\infty}(\Omega)$
coefficient.

\begin{theorem}\label{t:rate}
Let $\ell\approx\lvert\log H\rvert$ and assume
that \eqref{e:AHbounds} is satisfied.
Let $u$ solve \eqref{e:modelweak}
and let $\uHt\loc$ solve \eqref{e:mslump}.
Assume furthermore that the solution $u\smooth$ to
\eqref{e:PDEsmooth} belongs to $H^{1+s}(\Omega)$
for some $0<s\leq 1$.
Then,
\begin{equation*}
\| u - \uHt\loc\|_{L^2(\Omega)}
\lesssim
\Big(
H + 
H^s
\lvert \log H\rvert^{s+d/2}
\,\big(1+\eta(A_H\loc)\big)^s
\Big)
 \| f \|_{L^2(\Omega)}.
\end{equation*}
In particular, under the homogenization criterion from
Remark~\ref{r:homcrit}, a convergence rate is achieved.
If the domain is convex, then $s$ can be chosen as $s=1$,
i.e., the convergence
rate is linear up to a logarithmic factor.
\end{theorem}

\begin{proof}
This follows from combining
Proposition~\ref{p:priori} with \eqref{e:msHfestimate},
the triangle inequality and the
Friedrichs inequality.
If the domain is convex, elliptic regularity theory 
\cite{Grisvard1985,Dauge1988,Melenk2002} shows that
$s=1$ is an admissible choice.
\end{proof}

\begin{remark}\label{r:eta}
We emphasize that $\eta(A_H\loc)$ is not an error estimator for
the discretization error. It rather indicates whether the local
discrete model is appropriate. If $\eta(A_H\loc)$ is close
to zero, then the multiplicative constant on the right-hand side of the
formula in Theorem~\ref{t:rate} is of reasonable magnitude.
\end{remark}

\section{The periodic setting}\label{s:periodic}

In this section we justify the use of the local effective
coefficient $A_H$
in the periodic setting.
We show that the procedure 
in its idealized form with $\ell=\infty$ recovers the
classical periodic homogenization limit.
We denote by
$V:=H^1_{\#}(\Omega)/\mathbb R$ the space of periodic $H^1$ functions with
vanishing integral mean over $\Omega$.
We assume $\Omega$ to be a polytope allowing for 
periodic boundary conditions.
We adopt the notation of Section~\ref{s:nonloc},
in particular $W\subseteq V$ is the kernel of the quasi-interpolation
$I_H$,
$V_H$ is the space of piecewise affine globally continuous
functions of $V$,
and
$\Cor\loc$, $a$, $\tilde a\loc$, $\mathfrak a\loc$, $\mathcal A_H\loc$,
$A_H\loc$, $\mathcal K\loc$
are defined as in 
Section~\ref{s:nonloc} with the underlying space $V$ being
$H^1_{\#}(\Omega)/\mathbb R$.
We assume that the domain $\Omega$ matches with integer
multiples of the period.
We assume the triangulation $\T_H$ to match with the periodicity
pattern. For simplicial partitions this implies further symmetry
assumptions. In particular, periodicity with respect to a uniform
rectangular grid is not sufficient. Instead we require further
symmetry within the triangulated macro-cells,
see Example~\ref{ex:periodic} for an illustration.
This property will be required in the proof of
Propositon~\ref{p:homlimit} below.
In particular, not every periodic coefficient may meet this
requierement. Also, generating such a triangulation requires knowledge
about the length of the period.

\begin{example}\label{ex:periodic}
Figure~\ref{f:periodicity} displays a periodic coefficient
and a matching triangulation.\qed
\begin{figure}[t]
\centering
  \begin{tikzpicture}[scale=3.5]
   \draw[ultra thick,fill=gray!30] (0,0) -- (1,0) -- (1,1) -- (0,1) --cycle;
   \draw (0,.5)--(1,.5)
         (0,0)--(0,1)
         (.5,0)--(.5,1);
   \draw (0,.5)--(.5,1)
         (0,0)--(1,1)
         (.5,0)--(1,.5);
   \foreach \x/\y  in { .35/.1,
                        .85/.1,
                        .35/.6,
                        .85/.6
                                }
      {  \fill[color=blue] (\x,\y) circle (1pt);}
  \end{tikzpicture}
\hspace{2ex}
   \begin{tikzpicture}[scale=3.5]
   \draw[ultra thick,fill=gray!30] (0,0) -- (1,0) -- (1,1) -- (0,1) --cycle;
   \draw (0,.5)--(1,.5)
         (0,0)--(0,1)
         (.5,0)--(.5,1);
   \draw (0,.5)--(.5,1)
         (0,0)--(1,1)
         (.5,0)--(1,.5);
   \foreach \x/\y  in { .35/.1,
                        .85/.1,
                        .35/.6,
                        .85/.6,
                        .65/.4,
                        .15/.4,
                         .65/.9,
                        .15/.9 }
      {  \fill[color=blue] (\x,\y) circle (1pt);}
  \end{tikzpicture}
\caption{Periodic coefficients with respect to a square grid and
           triangulations: non-matching (left) and matching (right).
       \label{f:periodicity}}
\end{figure}
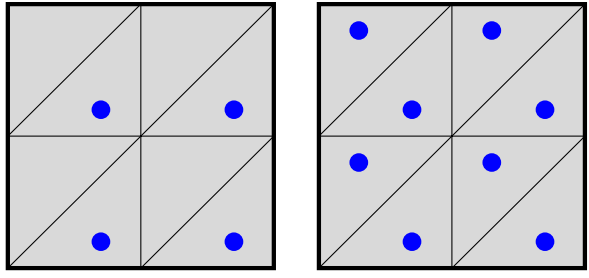
\end{example}

We remark that the error estimate
\eqref{e:mspgl2est} and Proposition~\ref{p:priori} hold
in this case as well.
Due to the periodic boundary conditions, the auxiliary solution 
$u\smooth$ utilized in the proof of Proposition~\ref{p:priori}
has the smoothness $u\smooth\in H^2(\Omega)$
so that those estimates are valid with $s=1$.
In the periodic setting, further properties of 
$A_H\loc$ can be
derived. 
First, it is not difficult to prove that the coefficient $A_H\loc$
is globally constant. The following result states that,
in the idealized case $\ell=\infty$, the coefficient $A_H\loc$
is even independent of the mesh-size $H$ and coincides with the 
classical homogenization limit, where for any $j=1,\dots,d$,
the corrector $\hat q_j\in H^1_\#(\Omega)/\mathbb R$ is the solution to
\begin{equation}\label{e:classhom}
\ddiv A(\nabla \hat q_j-e_j) =0
\text{ in }\Omega\text{ with periodic boundary conditions}.
\end{equation}

\begin{proposition}\label{p:homlimit}
Let $A$ be periodic and let $\T_H$ be
uniform and aligned with the periodicity
pattern of $A$
and let $V$, $W$ be spaces with periodic boundary conditions.
Then, for any $T\in\T_H$, the idealized coefficient
$A_H^{(\infty)}|_T$
coincides with the homogenized coefficient from the classical
homogenization theory.
In particular, $A_H^{(\infty)}$
is globally constant and independent of $H$.
\end{proposition}

\begin{proof}
Let $T\in\T_H$ and $j,k\in\{1,\dots,d\}$.
The definitions of $A_H^{(\infty)}|_T$
and $\mathcal K^{(\infty)}$ lead to
\begin{equation}\label{e:per_identity1}
\begin{aligned}
\fint_T A_{jk} \,dx- (A_H^{(\infty)}|_T)_{jk}
&
= |T|^{-1}\sum_{K\in\T_H} \int_K e_j\cdot (A\nabla q_{T,k}) \,dx
\\
&
= |T|^{-1} \int_\Omega e_j \cdot ( A\nabla q_{T,k})\,dx.
\end{aligned}
\end{equation}
The sum over all element correctors defined by
$q_k:=\sum_{T\in\T_H} q_{T,k}$ solves
\begin{equation}\label{e:gequation_multid}
 a(w,q_k) = (\nabla w, A e_k)_{L^2(\Omega)}
\quad\text{for all } w\in W.
\end{equation}
The definitions of $q_{T,k}$ and $q_k$ and the symmetry of $A$ lead to
\begin{equation}\label{e:per_identity2}
\begin{aligned}
|T|^{-1} \int_\Omega e_j \cdot ( A\nabla q_{T,k})\,dx 
&
= |T|^{-1} \int_\Omega \nabla q_j \cdot ( A\nabla q_{T,k}) \,dx
\\
&
= \fint_T e_k\cdot (A\nabla q_j)\,dx.
\end{aligned}
\end{equation}
Let $v\in V$. We have $(v-I_H v)\in W$ and therefore
by \eqref{e:gequation_multid} that
\begin{equation*}
\begin{aligned}
  \int_\Omega \nabla v \cdot ( A(\nabla q_j -e_j) )\,dx
&= \int_\Omega (\nabla I_H v) \cdot ( A(\nabla q_j -e_j) )\,dx
\\
&= \sum_{K\in\T_H} 
   \int_K (\nabla I_H v)\,dx \cdot \fint_K A(\nabla q_j -e_j)\,dx 
\end{aligned}
\end{equation*}
where for the last identity it was used that $\nabla I_H v$ is
constant on each element.
By periodicity we have that
$\fint_K A(\nabla q_j -e_j)\,dx =\fint_\Omega A(\nabla q_j -e_j)\,dx$
for any $K\in\T_H$.
Therefore, for all $v\in V$,
\begin{equation*}
\begin{aligned}
  \int_\Omega \nabla v \cdot (A (\nabla q_j -e_j))\,dx
&= \int_\Omega (\nabla I_H v)\,dx \cdot 
       \fint_\Omega A(\nabla q_j -e_j)\,dx =0
\end{aligned}
\end{equation*}
due to the periodic boundary conditions of $I_H v$.
Hence, the difference $\nabla q_j-e_j$
satisfies \eqref{e:classhom}.
This is the corrector problem from classical homogenization 
theory and, thus, the proof is concluded by the above formulae
\eqref{e:per_identity1}--\eqref{e:per_identity2}.
Indeed, by symmetry of $A$,
\begin{equation*}
(A_H^{(\infty)}|_T)_{jk}
=
\fint_T A_{jk} \,dx- \fint_T e_k \cdot (A\nabla q_j)\,dx 
=
\fint_T (e_j-\nabla q_j)\cdot A e_k \,dx.
\end{equation*}

\end{proof}

\begin{remark}
For Dirichlet boundary conditions, the method is different from
the classical periodic homogenization as it takes the boundary
conditions into account.
\end{remark}

The remaining parts of this section are devoted to an 
$L^2$ error estimate for the classical homogenization limit.
Let the coefficient $A=A_\varepsilon$ be periodic, oscillating
on the scale $\varepsilon$.
Let $H$ be the observation
scale represented by the mesh-size of the finite element mesh.
We couple $H$ to $\varepsilon$ so that the ratio $H/\varepsilon$
is constant.
Recall from Proposition~\ref{p:homlimit} that
the idealized coefficient $A_H^{(\infty)}=A_0$ for a constant coefficient
$A_0$ that is independent of $H$.
It is known (see, e.g., \cite{Allaire1997}) that,
in the present case of a symmetric coefficient, $A_0$ 
satisfies the bounds \eqref{e:AHbounds}.
Denote, for any $\varepsilon$, by $u_\varepsilon\in V$
the solution to
\begin{equation}\label{e:ueps}
\int_\Omega\nabla {u_\varepsilon}\cdot ( A_\varepsilon \nabla v) \,dx =
 F(v)\quad\text{for all } v\in V.
\end{equation}
Denote by $u_0\in V$ the solution to
\begin{equation}\label{e:u0}
\int_\Omega \nabla u_0 \cdot (A_0\nabla v)\,dx = F(v) 
 \quad\text{for all }v\in V.
\end{equation}
In periodic homogenization theory,
the function $u_0$ is called the homogenized solution.
The aim is to estimate $\|u_0-u_\varepsilon\|_{L^2(\Omega)}$
in terms of $\varepsilon$.
The following perturbation result is required.
\begin{lemma}[perturbed coefficient]\label{l:perturb}
 Let $H$ and $\varepsilon$ be coupled so that
 $H/\varepsilon$ is constant.
 Let the localization parameter $\ell$ be chosen of order
 $\ell\approx \lvert\log H\rvert$. Then,
 \begin{equation*}
  \| A_H^{(\infty)} - A_H\loc\|_{L^\infty(\Omega)} \lesssim H.
 \end{equation*}
 There exist $\varepsilon_0>0$ and $0<\alpha'\leq\beta'<\infty$
 such that for all
 $\varepsilon\leq\varepsilon_0$
\begin{equation*}
\alpha'|\xi|^2
\leq \xi \cdot ( A_H\loc(x)\xi)
\leq \beta' |\xi|^2
\end{equation*}
for all $\xi\in\mathbb{R}^d$ and almost all $x\in\Omega$. 
\end{lemma}
\begin{proof}
Remark~\ref{r:homo1} shows that $A_H^{(\infty)}$
and $A_H\loc$ are given on any $T\in\T_H$ through
\begin{equation*}
\begin{aligned}
(A_H^{(\infty)})_{j,k}|_T
& =
|T|^{-1} \int_{\Omega} e_j\cdot
                       \big(   A (\chi_T e_k-\nabla q_{T,k}) \big)
\\
\text{and }
\qquad
(A_H\loc)_{j,k}|_T
& =
|T|^{-1} \int_{\Omega_T} e_j\cdot
                       \big(   A (\chi_T e_k-\nabla q_{T,k}\loc) \big)
\end{aligned}
\end{equation*}
for any $j,k\in\{1,\dots,d\}$.
Thus,
\begin{equation*}
\begin{aligned}
\lvert (A_H^{(\infty)})_{j,k}|_T - (A_H\loc)_{j,k}|_T \rvert
& =
|T|^{-1} \int_{\Omega} e_j\cdot
                       \big(   A (\nabla (q_{T,k}-q_{T,k}\loc)) \big)
\\
&
\leq
|T|^{-1} \|A^{1/2}e_j\|_{L^2(\Omega)}
         \|A^{1/2} \nabla (q_{T,k}-q_{T,k}\loc)\|_{L^2(\Omega)} 
\\
&
\leq
|T|^{-1} \|A^{1/2} \nabla (q_{T,k}-q_{T,k}\loc)\|_{L^2(\Omega)}  .
\end{aligned}
\end{equation*}
It is shown in \cite[proof of Cor.~4.11]{HenningPeterseim2013} that
\begin{equation*}
\|A^{1/2} \nabla (q_{T,k}-q_{T,k}\loc)\|_{L^2(\Omega)}
\lesssim 
\exp(-c\ell) |T|^{1/2} .
\end{equation*}
In conclusion, the choice 
$ \ell \approx \lvert \log H \rvert $
implies the first stated estimate.
The second stated result follows from a perturbation argument 
because it is known \cite{Allaire1997} that
$A_H^{(\infty)}=A_0$ satisfies \eqref{e:Abounds}.
\end{proof}
The following result recovers the classical homogenization limit
$u_\varepsilon\to u_0$
strongly in $L^2$ as $\varepsilon\to 0$.
In particular, it quantifies the convergence speed and
states that for 
$f\in L^2(\Omega)$ an almost linear
rate is achieved.

\begin{proposition}[quantified homogenization limit]
Let $\Omega$ be convex,
let $u_\varepsilon\in V$ solve \eqref{e:ueps} and let
$u_0\in V$ solve \eqref{e:u0}.
For any $\varepsilon\leq\varepsilon_0$ 
(for  $\varepsilon_0$ from Lemma~\ref{l:perturb})
we have
\begin{equation*}
\|u_\varepsilon-u_0\|_{L^2(\Omega)}
\lesssim
\varepsilon
  \lvert\log \varepsilon\rvert^{1+d/2}
\|f\|_{L^2(\Omega)}.
\end{equation*}
\end{proposition}
\begin{proof}
As before, we couple $H$ and $\varepsilon$ such that
$H/\varepsilon$ is constant.
We denote by 
$u_H\loc\in V_H$ the solution to \eqref{e:mspgELL},
by $\uHt\loc\in V_H$ the solution to \eqref{e:mslump},
and by $\uHt^{(\infty)}\in V_H$ the solution to \eqref{e:mslump}
with the choice $\ell=\infty$ where in all problems
$A$ is replaced by $A_\varepsilon$.
Note that Lemma~\ref{l:perturb} implies stability of the discrete
system \eqref{e:mslump} and thereby unique existence of
$\uHt\loc$.
We employ the triangle inequality to split the error as follows
\begin{equation}\label{e:per_split}
\begin{aligned}
\|u_\varepsilon-u_0\|_{L^2(\Omega)}
\lesssim
&
\|u_\varepsilon-u_H\loc\|_{L^2(\Omega)}
+
\|u_H\loc-\uHt\loc\|_{L^2(\Omega)}
\\
&
+
\|\uHt\loc-\uHt^{(\infty)}\|_{L^2(\Omega)}
+
\|\uHt^{(\infty)}-u_0\|_{L^2(\Omega)} .
\end{aligned}
\end{equation}
Estimate \eqref{e:mspgl2est} allows to bound the first term
on the right-hand side of \eqref{e:per_split} as
\begin{equation*}
\|u_\varepsilon-u_H\loc\|_{L^2(\Omega)}
\lesssim
\varepsilon \|f\|_{L^2(\Omega)}.
\end{equation*}
The second term on the right-hand side of \eqref{e:per_split}
was bounded in Proposition~\ref{p:priori}.
With the Friedrichs inequality the result reads
\begin{equation*}
\|u_H\loc - \uHt\loc\|_{L^2(\Omega)}
\lesssim
\varepsilon \lvert \log \varepsilon\rvert^{1+d/2}  \| f \|_{L^2(\Omega)}
\end{equation*}
where it was used that
$\eta(A_H\loc)=0$ because $A_H\loc$ is spatially constant.
In order to bound the third term on the right-hand side of
\eqref{e:per_split} we use the the stability of the discrete problems
and the perturbation result of Lemma~\ref{l:perturb} to deduce
\begin{equation*}
\|\uHt\loc-\uHt^{(\infty)}\|_{L^2(\Omega)}
\lesssim
\| A_H^{(\infty)} - A_H\loc\|_{L^\infty(\Omega)}\| f \|_{L^2(\Omega)}
\lesssim
\varepsilon \| f \|_{L^2(\Omega)}.
\end{equation*}
For the fourth term on the right-hand side of \eqref{e:per_split}
it is enough to note that $\uHt^{(\infty)}$ is the Galerkin
approximation of $u_0$ in $V_H$, which satisfies
\begin{equation*}
\|\uHt^{(\infty)}-u_0\|_{L^2(\Omega)}
\lesssim 
\varepsilon^2 \|f\|_{L^2(\Omega)}
\end{equation*}
on convex domains.
The combination of the foregoing estimates concludes the proof.
\end{proof}

\section{Numerical illustration}\label{s:num}

In section, we present numerical experiments
on the unit square domain $\Omega =(0,1)^2$
with homogeneous Dirichlet boundary conditions.
We consider the
following worst-case error (referred to as the $L^2$ error)
as error measure
$$
   \sup_{f\in L^2(\Omega)\setminus\{0\}}
   \frac{\| u(f) - u_{\mathrm{discrete}}(f) \|_{L^2(\Omega)}}
        {\|f\|_{L^2(\Omega)}}
$$
where $u(f)$ is the exact solution to \eqref{e:modelweak}
with right-hand side $f$
and $u_{\mathrm{discrete}}(f)$ a discrete approximation
(standard FEM or local effective coefficient or
quasi-local effective coefficient or
$L^2$-best approximation).
The error quantity is approximated by solving an eigenvalue
problem on the reference mesh.

\subsection{First experiment: Convergence rates}

Consider the scalar coefficient $A$ 
$$
A
(x_1,x_2)=
\left(
\frac{11}{2}
+\sin\left(\frac{2\pi}{\varepsilon_1}x_1\right) 
        \sin\left(\frac{2\pi}{\varepsilon_1} x_2\right)
+4\,\sin\left(\frac{2\pi}{\varepsilon_2} x_1\right)
        \sin\left(\frac{2\pi}{\varepsilon_2} x_2\right)
\right)^{-1}
$$
with $\varepsilon_1=2^{-3}$ and $\varepsilon_2=2^{-5}$.
We consider a sequence of uniformly refined meshes of mesh
size
$H=\sqrt{2}\times 2^{-1},\dots,\sqrt{2}\times 2^{-6}$.
The corrector problems are solved on a reference mesh of 
width $h=\sqrt{2}\times 2^{-9}$.
The localization (or oversampling) parameter 
is  chosen as $\ell=2$.
Figure~\ref{f:coeff1} displays the coefficient $A$.
The four components of the reconstructed coefficient $A_H\loc$
for $H=\sqrt{2}\times 2^{-6}$
are displayed in Figure~\ref{f:reconstr_coeff1}.
Figure~\ref{f:convergence} compares the $L^2$ errors of
the standard FEM,
the FEM with the local effective coefficient,
the method with the quasi-local effective coefficient,
and the $L^2$-best approximation in dependence of $H$.
For comparison, also the error of the Multiscale Finite Element
Method (MSFEM) from \cite{EfendievHou2009} is displayed.
As expected, the error of the FEM is of order $\mathcal O(1)$
because the coefficient is not resolved by the mesh-size $H$.
The error for the quasi-local effective coefficient
is close to the best-approximation. 
The local effective coefficient leads to comparable errors
on coarse meshes. On the finest mesh, where the coefficient is
almost resolved, the error deteriorates. This effect,
referred to as ``resonance effect'', will be studied in the 
second numerical experiment.
Table~\ref{t:bounds} lists the values of the
estimator
$\eta(A_H\loc)$
as well as the bounds $\alpha_H$ and $\beta_H$ on
$(A_H\loc)$.
The estimator $\eta(A_H\loc)$ is small on the first meshes,
which corresponds to an effective coefficient close
to a constant. The estimator increases for the meshes
approaching the resonance regime.
The values of the coefficient $A$ range in the interval
$[\alpha,\beta]=[0.096,1.55]$. In this example, the discrete
bounds $\alpha_H$, $\beta_H$ stay in this interval.

\begin{figure}
\centering
\includegraphics{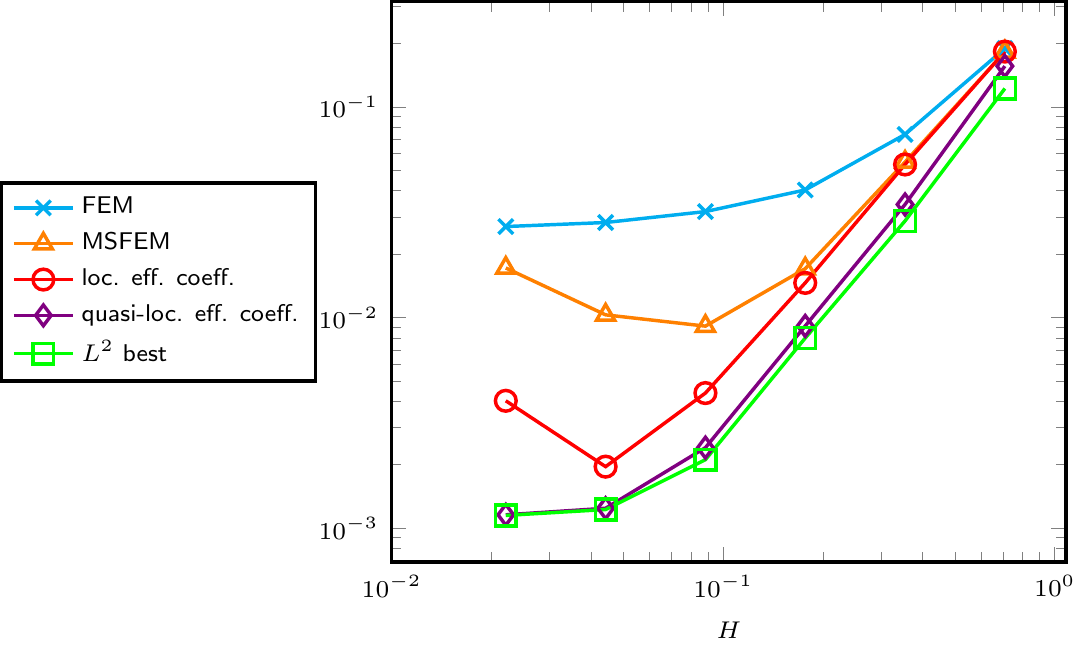}
\caption{Convergence history under uniform mesh refinement.
                 \label{f:convergence}}
\end{figure}

\begin{figure}
\centering
\includegraphics[height=.33\textheight]{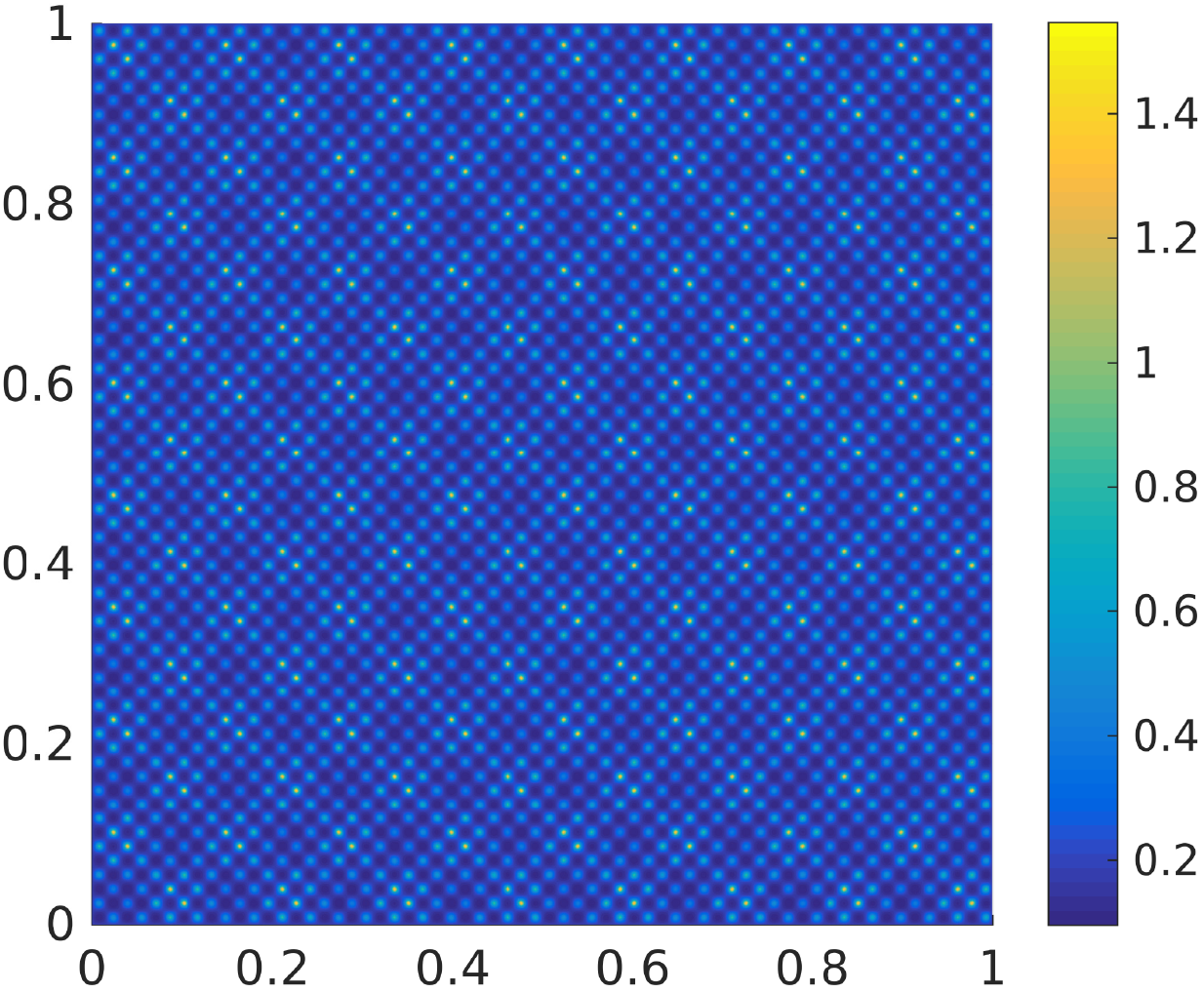}
\caption{The scalar coefficient $A$ for the first experiment.%
         \label{f:coeff1}}
\end{figure}

\begin{figure}
\centering
\includegraphics[height=.33\textheight]{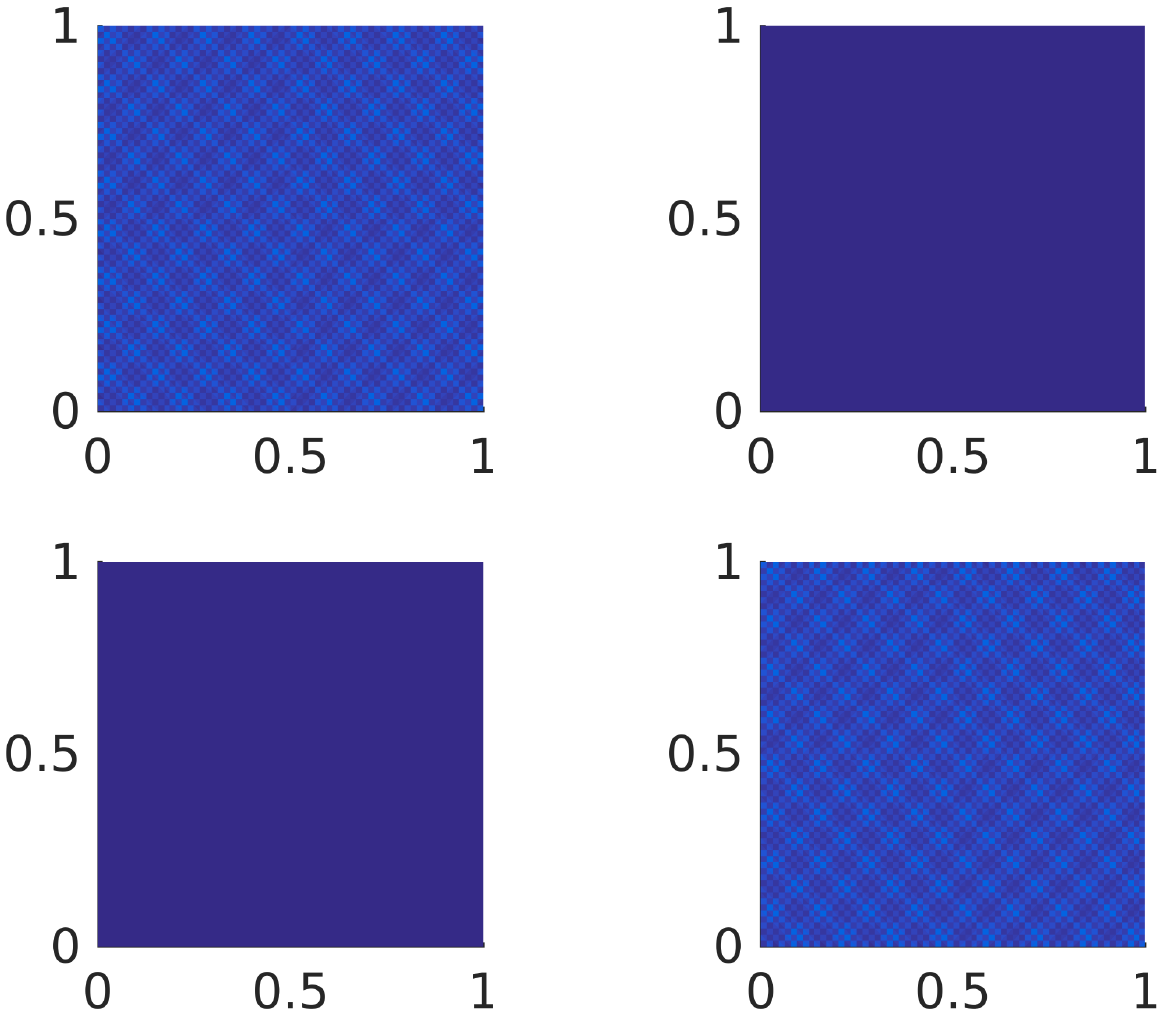}
\hspace{1cm}
\includegraphics[height=.33\textheight]{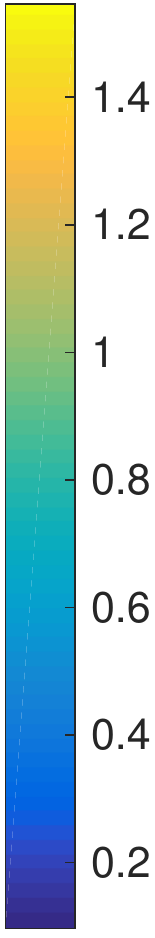}
\caption{Matrix entries of the  reconstructed
           localized coefficient $(A_H\loc)$ in the first experiment
            for $H=\sqrt{2}\times 2^{-6}$.\label{f:reconstr_coeff1}}
\end{figure}

\begin{table}
\centering
\begin{tabular}{l||l|l|l}
$H$    &  $\eta(A_H\loc)$    &  $\alpha_H$    &   $\beta_H$ \\
\hline
$\sqrt{2}\times2^{-1}$  & 3.2108e-02  &  1.9223e-01 &  2.0786e-01 \\
$\sqrt{2}\times2^{-2}$  & 1.1267e-02  &  1.9568e-01 &  1.9954e-01 \\
$\sqrt{2}\times2^{-3}$  & 1.4765e-02  &  1.9579e-01 &  1.9986e-01 \\
$\sqrt{2}\times2^{-4}$  & 5.3952e-01  &  1.8323e-01 &  2.1992e-01 \\
$\sqrt{2}\times2^{-5}$  & 1.7199e+00  &  1.6909e-01 &  2.3257e-01 \\
$\sqrt{2}\times2^{-6}$  & 1.5538e+01  &  1.4070e-01 &  3.0277e-01 \\
\end{tabular}
\caption{%
Values of the estimator
$\eta(A_H\loc)$
and the bounds $\alpha_H$ and $\beta_H$ on
$A_H$ for the first experiment.
The values of the coefficient $A$ range in the interval
$[\alpha,\beta]=[0.096,1.55]$.\label{t:bounds}}
\end{table}

\subsection{Second experiment: Resonance effects}

In this experiment we investigate so-called resonance effects
of our homogenization procedure.
These effects occur because, unlike in Section~\ref{s:periodic},
in the present case we deal with Dirichlet boundary conditions as
well as meshes that do not satisfy requirements in the spirit of
Example~\ref{ex:periodic}.
We consider a fixed mesh of width
$H=\sqrt{2}\times2^{-4}$
and the scalar coefficient
$$
A
(x_1,x_2)=
\left(
5
+
4\,\sin\left(\frac{2\pi}{\varepsilon} x_1\right)
        \sin\left(\frac{2\pi}{\varepsilon} x_2\right)
\right)^{-1}
$$
for a sequence of parameters
$\varepsilon =2^0,2^{-1},\dots,2^{-6}$.
The coefficient $(A_H\loc)$ was computed with the same reference
mesh and the same oversampling parameter as in the first
experiment.
Figure~\ref{f:resonance} displays the $L^2$ errors
normalized by the $L^2$ error of the $L^2$-best approximation.
On the third mesh, where $H$ and $\varepsilon$ have the same order
of magnitude, the local effective coefficient leads
to a larger error compared to the coarser meshes
(where the coefficient is resolved by $H$)
and finer meshes,
where $H$ is much coarser than $\varepsilon$ and the
effective coefficient is close to a constant.
We observe that the values of the estimator $\eta(A_H\loc)$
are large in the resonance regime where also the error of
the method the local effective coefficient is large.
For smaller values of $\varepsilon$, the values of 
$\eta(A_H\loc)$ are close to zero, which indicates that 
the homogenization criterion from Remark~\ref{r:homcrit} is satisfied,
cf.\ also Remark~\ref{r:eta}.

\begin{figure}
\centering
\includegraphics{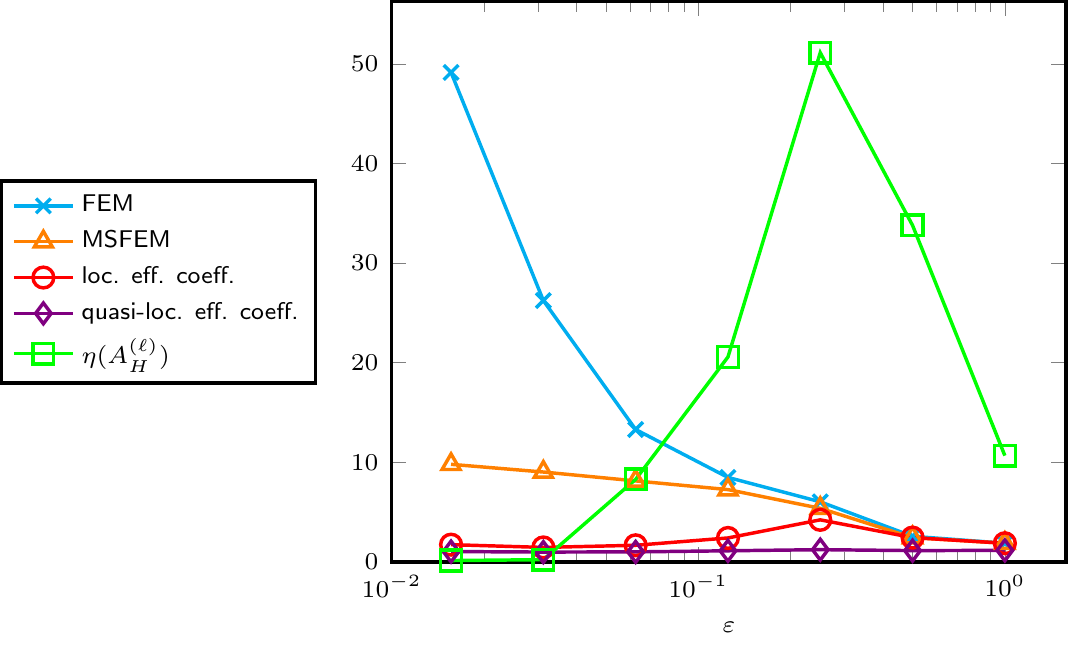}
\caption{Resonance effect:
Normalized (by $L^2$-best error) errors of FEM,
local effective model and quasi-local effecitve model;
and values of the estimator $\eta(A_H\loc)$.\label{f:resonance}}
\end{figure}

{\footnotesize
\bibliographystyle{alpha}
\bibliography{bib_homcoeff_mms}
}

\end{document}